\documentclass[12pt]{amsart}

\usepackage{enumerate}
\usepackage{tikz-cd}
\usepackage{color}
\usepackage[all]{xy}
\usepackage{amsmath, amssymb}
\usepackage{graphics}
\usepackage{graphicx} 
\usepackage{mathtools}
\usepackage[pdftex,hypertexnames=false]{hyperref}
\hypersetup{
	colorlinks=true,
	citecolor=blue,
	linkcolor=blue,
	urlcolor=blue
	}
	\usepackage{cleveref}
	\usepackage[square,numbers]{natbib}
	\usepackage{autonum}
\usepackage{geometry}
\numberwithin{equation}{subsection}
\theoremstyle{plain}
\newtheorem{theorem}{Theorem}[section]
\crefname{theorem}{Theorem}{Theorems}

\crefname{mainresult}{Main Result}{Main Results}
\newtheorem{lemma}[theorem]{Lemma}
\crefname{lemma}{Lemma}{Lemmas}
\newtheorem{proposition}[theorem]{Proposition}
\crefname{proposition}{Proposition}{Propositions}
\newtheorem{corollary}[theorem]{Corollary}
\crefname{corollary}{Corollary}{Corollaries}

\crefname{conjecture}{Conjecture}{Conjectures}
\theoremstyle{definition}
\newtheorem{definition}[theorem]{Definition}
\crefname{definition}{Definition}{Definitions}
\newtheorem{example}[theorem]{Example}
\crefname{example}{Example}{Examples}
\newtheorem{remark}[theorem]{Remark}
\crefname{remark}{Remark}{Remarks}

\crefname{criterion}{Criterion}{Criterions}

\crefname{problem}{Problem}{Problems}

\crefname{assumption}{Assumption}{Assumptions}
\newtheorem{notation}{Notation}
\crefname{notation}{Notation}{Notations}

\crefname{question}{Question}{Questions}

\newcommand{\End}{\operatorname{End}}
\newcommand{\Hom}{\operatorname{Hom}}

\newcommand{\rad}{\operatorname{rad}}

\newcommand{\soc}{\operatorname{soc}\nolimits}

\newcommand{\add}{\operatorname{add}}
\renewcommand{\mod}{\mathrm{mod}\,}
\newcommand{\Pfin}{\mathcal{P}^{<\infty}}
\newcommand{\Ifin}{\mathcal{I}^{<\infty}}
\newcommand{\cotilt}{\operatorname{cotilt}}

\newcommand{\proj}{\operatorname{proj}}

\newcommand{\ind}{\operatorname{ind}}

\newcommand{\F}{\mathbb{F}}

\newcommand{\Ext}{\operatorname{Ext}}

\newcommand{\findim}{\operatorname{fin.dim}}
\newcommand{\gldim}{\operatorname{gl.dim}}
\newcommand{\injdim}{\operatorname{inj.dim}}
\newcommand{\projdim}{\operatorname{proj.dim}}
\newcommand{\domdim}{\operatorname{dom.dim}}

\renewcommand{\L}{\Lambda}

\newcommand{\B}{\mathcal{B}}
\newcommand{\C}{\mathcal{C}}

\newcommand{\op}{{\rm op}}

\begin{document}
\title[Cohen-Macaulay truncated path and Nakayama algebras]{Cohen-Macaulay truncated path and Nakayama algebras}

\author[T. Honma]{Takahiro Honma}
\address{%
{\rm (T.H.)} Department of General Education \\
National Institute of Technology (KOSEN), Yuge College \\ 
Ochi \\
Ehime 794-2593\\ 
Japan}
\email{99cfqqc9@gmail.com}
\author[S. Usui]{Satoshi Usui$^*$}
\address{%
{\rm (S.U.)} Department of General Education \\
Tokyo Metropolitan College of Industrial Technology \\ 
8-17-1, Minamisenju, Arakawa-ku \\
Tokyo 116-8523\\ 
Japan}
\email{usui@metro-cit.ac.jp}

\thanks{$^*$Corresponding author}

\subjclass[2020]{Primary 16E65; 
Secondary 
16E05,   
16E10,  
16G20
}

\keywords{Cohen-Macaulay algebra,
Iwanaga-Gorenstein algebra, 
truncated path algebra,
Nakayama algebra}

\date{}

\dedicatory{}

\begin{abstract}
This paper characterizes Cohen-Macaulay algebras in the sense of Auslander and Reiten [{\it Progr.~Math.}~95(1991)] within the following two classes of finite dimensional algebras: (i) truncated path algebras and (ii) Nakayama algebras.
For the former, our characterization in terms of ordinary quivers shows that if a Cohen-Macaulay truncated path algebra $\L$ is connected, then $\L$ or $D(\L)$ gives a dualizing module, where $D$ denotes the standard duality. 
For the latter, we show that every Cohen-Macaulay Nakayama algebra $\L$ is Iwanaga-Gorenstein.
In particular, $\L$ always gives a dualizing module.
Consequently, there is no non-trivial Cohen-Macaulay algebra among truncated path and Nakayama algebras.
\end{abstract}

\maketitle

\section{Introduction}
\label{Introduction}

Let $R$ be a commutative Cohen-Macaulay local ring of finite Krull dimension with a dualizing module (also called a canonical module) $\omega$.
Then Sharp \cite[Theorem 3.11]{Sharp_1972} proved that there are quasi-inverse equivalences
\begin{align} \label{eq_1}
    -\otimes_R \omega:\Pfin(R) \to \Ifin(R) \quad\text{and}\quad \Hom_R(\omega,-):\Ifin(R) \to \Pfin(R), 
\end{align}
where $\Pfin(R)$ (resp.~$\Ifin(R)$) denotes the category of finitely generated $R$-modules of finite projective (resp.~injective) dimension.
Later, as a non-commutative analogue of commutative Cohen-Macaulay rings, Auslander and Reiten \cite{Auslander-Reiten_1991} introduced {\it Cohen-Macaulay algebras}, namely, artin algebras $\L$ admitting a $\L$-bimodule $W$, called a {\it dualizing module}, which induces the same quasi-inverse equivalences as in (\ref{eq_1}).
It is known that a dualizing module is unique up to isomorphism as a one-sided module but not as a bimodule; see \cref{remark_1}.
Also, there are other non-commutative analogues of commutative Cohen-Macaulay rings; see \cite{Beligiannis-Reiten_2007,Bahlekeh-Salarian-Toosi_2023,Goto-Nishida_2001}.

As in the case of commutative rings, Iwanaga-Gorenstein algebras \cite{Enochs-Jenda_Book},  which form a non-commutative analogue of commutative Gorenstein rings, are  precisely the Cohen-Macaulay algebras $\L$ for which the $\L$-bimodule $\L$ gives a dualizing module.
Besides, we have the following examples of Cohen-Macaulay algebras, 
where the first is due to Auslander and Reiten \cite{Auslander-Reiten_1991}, the second to Nishida \cite{Nishida_2002}, and the last two to Chan, Iyama and Marczinzik \cite{Chan-Iyama-Marczinzik_2024_Cohen-Macaulay}:
\begin{enumerate}
\item  artin algebras $\L$ such that the finitistic dimensions of $\L$ and $\L^\op$ are zero;
\item  the artin algebra $\overline{\L}$ given in \cite[Example 1]{Nishida_2002};
\item  
contracted preprojective algebras of Dynkin type; 
\item tensor products of finite dimensional Cohen-Macaulay algebras over a field $K$.
\end{enumerate}
These seem to be all known examples so far.
The algebras in (1) contain local algebras and are precisely the Cohen-Macaulay algebras $\L$ for which the $\L$-bimodule $D(\L)$ gives a dualizing module, where $D$ denotes the Matlis duality.
Therefore, it is natural to try to find Cohen-Macaulay algebras $\L$ for which neither $\L$ nor $D(\L)$ gives a dualizing module, more essentially, those for which a dualizing module, regarded as a one-sided module, admits a direct summand that is neither projective nor injective.
In fact, the algebra in (2) and some of the algebras in (3) are such non-trivial Cohen-Macaulay algebras.

In this paper, we characterize the Cohen-Macaulayness of two classes of finite dimensional algebras.

First, motivated by the example of Chan, Iyama and Marczinzik \cite[Example 1.3(1)]{Chan-Iyama-Marczinzik_2024_Cohen-Macaulay} showing that radical square zero truncated path algebras $KQ/\langle\text{all paths of length $2$}\rangle$ of quivers $Q$ without sink or source are Cohen-Macaulay algebras in (1), we focus on the truncated path algebras $\L=KQ/\langle\text{all paths of length $L+1$}\rangle$ of connected quivers $Q$, where $L\geq 1$.
After introducing the notion of {\it balanced quivers} (\cref{def_balanced}), we show in \cref{claim_4} that  $\L$ is Cohen-Macaulay if and only if $Q$ is balanced.
In this case, $\L$ or $D(\L)$ gives a dualizing module.
Thus, there is no non-trivial Cohen-Macaulay truncated path algebra.
This is our first main result.
Note that the connectedness of $Q$ is used only to simplify the statement; see \cref{claim_10} for the case where $Q$ is not necessarily connected.
Also, Chan, Iyama and Marczinzik \cite{Chan-Iyama-Marczinzik_2024_Cohen-Macaulay} introduced 
a class of Cohen-Macaulay algebras, called 
{\it minimal Auslander-Cohen-Macaulay algebras}, as a generalization of minimal Auslander-Gorenstein algebras \cite{Iyama-Solberg_2018}.
In \cref{claim_11}, we show that among truncated path algebras, minimal Auslander-Cohen-Macaulay algebras are precisely the minimal Auslander-Gorenstein Nakayama algebras.

Second, we consider Nakayama algebras. 
As our second main result, we show in \cref{claim_1} that a Nakayama algebra is Cohen-Macaulay if and only if it is Iwanaga-Gorenstein. 
Thus, there is no non-Iwanaga-Gorenstein Cohen-Macaulay Nakayama algebra; in other words, we can always choose $\L$ as a dualizing module over a Cohen-Macaulay Nakayama algebra $\L$.

Consequently, there is no non-trivial Cohen-Macaulay algebra among truncated path and Nakayama algebras.

This paper is organized as follows.
We first recall some basic notions and results related to Cohen-Macaulay algebras in \Cref{Preliminaries}.
We then prove our main results in \Cref{Main_results}.

\subsection*{Use of AI}
GPT-5.6(Sol) was used to improve the presentation of the paper. 



\section{Preliminaries} \label{Preliminaries} 

This section is devoted to recalling some definitions and basic results concerning Cohen-Macaulay algebras that will be needed later.

We first fix some notations and conventions.
Throughout this paper, all modules are assumed to be finitely generated right modules.
For a ring $\L$, denote by $\mod\L$ the category of $\L$-modules and by $\gldim\L$ its global dimension.
For $M\in\mod\L$, its projective and injective dimensions are denoted by $\projdim_\L M$ and $\injdim_\L M$, respectively. 
Assume that all subcategories are full and closed under isomorphisms.
For a Krull-Schmidt category $\C$, denote by $\ind\C$ the set of isomorphism classes of indecomposable objects in $\C$.
Denote by $|X|$ the cardinality of a set $X$.

\subsection{Cotilting modules} 

Since there is a tight connection between dualizing modules and  cotilting modules, we begin by recalling the definition of cotilting modules.
In this and the next subsections, all algebras are assumed to be artin algebras over a commutative artin ring $R$.

Let $\L$ be an algebra.
For a subcategory $\C$ of $\mod\L$ and a module $M\in\mod\L$, a {\it left $\C$-approximation} of $M$ is a morphism $f : M \to C$ in $\mod\L$ with $C\in\C$ such that every morphism $f^\prime : M\to C^\prime$ with $C^\prime\in\C$ factors through $f$.
We say that $\C$ is {\it covariantly finite} if every $\L$-module admits a left $\C$-approximation. 
We say that a morphism $f: M\to N$ is {\it left minimal} if every endomorphism $g: N \to N$ satisfying $gf=f$ is an isomorphism.
A left minimal left $\C$-approximation of $M$ is simply called a {\it  minimal left $\C$-approximation of $M$}.
If $M$ admits a left $\C$-approximation, then there uniquely exists a minimal left $\C$-approximation $f: M  \to C$; see \cite[Proposition 1.1]{Auslander_Reiten_1991_applications}.
We say that $\C$ is {\it coresolving} if it contains all injective $\L$-modules and is closed under extensions and cokernels of monomorphisms.
Denote by $\widehat\C$ the subcategory of $\L$-modules $N$ such that there is an exact sequence $0 \to X_n \to \cdots \to X_0 \to N \to 0$ with $X_i \in \C$.
Denote by $\Ifin(\L)$ the subcategory of $\L$-modules having finite injective dimension. 
Define $\Pfin(\L)$  dually.

We say that a module $T\in\mod\L$ is {\it cotilting} \cite{Miyashita_1986} if $T\in\Ifin(\L)$, $\Ext^i_\L(T, T) = 0$ for all $i>0$, and $D(\L)\in\widehat{\add T}$, where $D=\Hom_R(-, E)$ denotes the Matlis duality, where $E$ is the injective envelope of $R/\rad R$.  
By \cite[Corollary 1 on p.124]{Miyashita_1986}, we have $|\ind\add T|=|\ind\add \L|$ for all cotilting $\L$-modules $T$, where $\add M$ denotes the category of direct summands of direct sum of finite copies of a $\L$-module $M$. 
Moreover, by \cite[Theorem 1.5]{Miyashita_1986}, ${}_{\End_\L(T)}T$ is a cotilting module, and the natural algebra morphism $\L \to \End({}_{\End_\L(T)}T)^\op$, given by $x\mapsto[t\mapsto tx]$, is an isomorphism.  
Let $\cotilt\L$ be the set of equivalence classes of cotilting $\L$-modules with respect to the equivalence relation given by $T \sim T^\prime \Leftrightarrow \add T = \add T^\prime$. 
By \cite[Theorem 5.5]{Auslander_Reiten_1991_applications}, there is a bijection between $\cotilt\L$ and the set of covariantly  finite coresolving subcategories of $\mod\L$ contained in $\Ifin(\L)$, where we associate to $T\in\cotilt\L$ the subcategory $\widehat{\add T}$. 
There is a partial order on $\cotilt\L$ given by $T \geq U \Leftrightarrow \Ext^i_\L(T, U) = 0$ for all $i>0$.
Clearly, the injective cogenerator $D(\L)$ is the minimum element of $\cotilt\L$.
The following proposition characterizes the existence of the maximum element of $\cotilt\L$.

\begin{proposition}
\label{Preliminaries_1}
The following conditions are equivalent for $T\in\cotilt\L$.
\begin{enumerate}
    \item $T$ is a maximal element of $\cotilt\L$.
    \item $T$ is the maximum element of $\cotilt\L$.
    \item $\widehat{\add T} = \Ifin(\L)$.
    \item $\Ifin(\L)$ is covariantly finite. 
    \item Every simple $\End_\L(T)^\op$-module is a direct summand of the semisimple quotient $T/\rad T$ of the cotilting $\End_\L(T)^\op$-module $T$ by its Jacobson radical $\rad T$.
\end{enumerate}    
In this case, we have $\add T = \add X$, where $X$ denotes the minimal left $\Ifin(\L)$-approximation of $\L_\L$.
\end{proposition}
\begin{proof}
The equivalence (3) $\Leftrightarrow$ (4) is clear.
The equivalences (1) $\Leftrightarrow$ (2)  $\Leftrightarrow$ (3) and (3) $\Leftrightarrow$ (5) follow from \cite[dual of Proposition 3.2]{Iyama-Zhang_2019} and \cite[dual of Proposition 7.1]{Dugas-Huisgen--Zimmermann_2011}, respectively. 
The last statement follows from \cite[dual of Supplement II]{Huisgen-Zimmermann_2021}.
\end{proof}

Following Chan, Iyama and Marczinzik \cite{Chan-Iyama-Marczinzik_2024_Cohen-Macaulay}, we call a cotilting module satisfying the equivalent conditions in  \cref{Preliminaries_1} {\it Ext-maximal}. 
Note that Auslander and Reiten  \cite{Auslander_Reiten_1991_applications,Auslander-Reiten_1991} call Ext-maximal cotilting modules {\it strong cotilting modules}.

Recall that the {\it finitistic dimension} of an algebra $\L$ is defined as
\begin{align}
    \findim\L:=\sup\{\projdim_\L X \mid X \in \Pfin(\L)\}.
\end{align}
From the definition, we have $\findim\L^\op=\sup\{\injdim_\L X \mid X \in \Ifin(\L)\}$.
We end this subsection with the following consequence of  \cref{Preliminaries_1}.

\begin{proposition}[{\cite[Example 1.3(1)]{Chan-Iyama-Marczinzik_2024_Cohen-Macaulay}}]\label{Chan-Iyama-Marczinzik_2024_Cohen-Macaulay_Example1.3(1)}
The following conditions are equivalent for an algebra $\L$.
\begin{enumerate}
    \item $D(\L)$ is an Ext-maximal cotilting $\L$-module.
    \item $\soc\L_\L$ contains all simple $\L$-modules. 
    \item $\findim\L^\op = 0$.
\end{enumerate}
\end{proposition} 



\subsection{Cohen-Macaulay algebras} 

We begin with the definition of Cohen-Macaulay algebras.

\begin{definition}[{\cite[p.223]{Auslander-Reiten_1991}}]\label{def_CM}
An algebra $\L$ is called {\it Cohen-Macaulay} if there exists a $\L$-bimodule $W$ which gives quasi-inverse equivalences
\begin{align}
    -\otimes_\L W:\Pfin(\L) \to \Ifin(\L) \quad\text{and}\quad \Hom_\L(W,-):\Ifin(\L) \to \Pfin(\L). 
\end{align}
We call the $\L$-bimodule $W$ a {\it dualizing module}.
\end{definition} 

From the definition,  a direct product $\L_1 \times\cdots\times\L_n$ of algebras $\L_i$ is Cohen-Macaulay if and only if each $\L_i$ is Cohen-Macaulay. 
Auslander and Reiten characterised dualizing modules in terms of cotilting modules.

\begin{proposition}[{\cite[Proposition 1.3]{Auslander-Reiten_1991}}]
\label{Auslander-Reiten_1991_1.3}
A $\L$-bimodule $W$ is a dualizing module if and only if $W$ is an Ext-maximal cotilting module as a left and as a right $\L$-module, and the natural algebra morphism $\L \to \End(W_\L)$, given by $x\mapsto [w\mapsto xw]$, is an isomorphism.
\end{proposition} 

\begin{remark} \label{remark_1}
By \cref{Preliminaries_1,Auslander-Reiten_1991_1.3}, a dualizing module is uniquely determined up to isomorphism as a left and as a right module. 
However, it is not as a bimodule; see \cite[p.226]{Auslander-Reiten_1991}.
\end{remark}

One of the nice properties of a Cohen-Macaulay algebra is as follows.

\begin{proposition}[{\cite[Proposition 1.6]{Auslander-Reiten_1991}}]\label{Auslander-Reiten_1991_1.6}
Let $\L$ be a Cohen-Macaulay algebra with dualizing module $W$.
Then $\findim\L= \injdim_\L W = \injdim_{\L^\op} W =\findim\L^\op$.
\end{proposition} 

An algebra $\L$ is called {\it Iwanaga-Gorenstein} \cite{Enochs-Jenda_Book} if $\L$ has finite injective dimension as a left and as a right $\L$-module.
Iwanaga-Gorenstein algebras are precisely the Cohen-Macaulay algebras $\L$ whose dualizing modules are given by $\L$. 
As the following proposition shows, the finiteness of the projective dimension of a dualizing module on one side is enough to characterize the Iwanaga-Gorensteinness of Cohen-Macaulay algebras.

\begin{proposition}\label{claim_8}
The following conditions are equivalent for  a Cohen-Macaulay algebra $\L$ with dualizing module $W$.
\begin{enumerate}
\item  $\L$ is an Iwanaga-Gorenstein algebra.
\item $\L$ has finite injective dimension as a left or as a right $\L$-module.
\item $W$ has finite projective dimension as a left or as a right $\L$-module.
\end{enumerate}
\end{proposition}
\begin{proof}
The implications (1) $\Rightarrow$ (2) and  (1) $\Rightarrow$ (3) are clear.
The implication (2) $\Rightarrow$ (1) follows from \cite[Proposition 6.10]{Auslander_Reiten_1991_applications} and  \cref{Auslander-Reiten_1991_1.6}.
To prove (3) $\Rightarrow$ (2), assume (3) holds.
Then, since $D(\L)\in\widehat{\add W}$ from the definition of cotilting modules, we can conclude that $D(\L)$  has finite projective dimension on one side.
\end{proof}

We end this section with the following consequence of  \cref{Chan-Iyama-Marczinzik_2024_Cohen-Macaulay_Example1.3(1)}.

\begin{proposition}[{Auslander-Reiten}]
\label{Chan-Iyama-Marczinzik_2024_Cohen-Macaulay_Example1.3(1)_corollary}
The following conditions are equivalent for an algebra $\L$.
\begin{enumerate}
\item $\L$ is a Cohen-Macaulay algebra with dualizing module $D(\L)$.
\item $\soc\L_\L$ contains all simple $\L$-modules, and $\soc{}_\L\L$ contains all simple $\L^\op$-modules.
\item $\findim\L = 0 = \findim\L^\op$.
\end{enumerate}
\end{proposition} 




\section{Main results} \label{Main_results}

In the rest of this paper, we assume that all algebras are finite dimensional algebras over a field $K$.
The aim of this section is to characterize the Cohen-Macaulayness of truncated path  and Nakayama algebras.

\subsection{Cohen-Macaulay truncated path algebras} 
\label{Cohen-Macaulay_truncated_path_algebras}

We begin by considering truncated path algebras.
For the basics on quivers and bound quiver algebras, we refer to \cite{ARS95,ASS06}. 
We always assume that  all quivers $Q$ are finite in the sense that the vertex set $Q_0$ and the arrow set $Q_1$ are both finite.

For $n\geq1$, a {\it path of length $n$} in a quiver $Q$ is a sequence $p=a_1\cdots a_n$ of arrows $a_i \in Q_1$ such that the target $t(a_i)$ of $a_i$ equals the source $s(a_{i+1})$ of $a_{i+1}$ for $1\leq i< n$.
Define $s(p):=s(a_1)$ and $t(p):=t(a_n)$.
For each $i\in Q_0$, the corresponding trivial path is denoted by $e_i$. 
Note that $s(e_i)=i=t(e_i)$.
Denote by $\B$ the set of all paths in $Q$.
A non-trivial path $c$ is called a {\it cycle} if $s(c)=t(c)$.
A vertex $i\in Q_0$ is called a {\it sink} if there is no arrow starting at $i$. We call $i\in Q_0$ {\it postcyclic} if there is a path from a vertex lying on a cycle to $i$. 
Define a {\it source} and a {\it precyclic vertex} dually.
We call $Q$ {\it acyclic} if it has no sink or source.

\begin{definition}\label{def_balanced}
A quiver is called {\it balanced} if the set of its postcyclic vertices coincides with the set of its precyclic vertices.
\end{definition}

Note that a quiver is balanced if and only if it has no postcyclic sink or precyclic source. 
We also note that every connected component of a balanced quiver is either an acyclic quiver or a quiver without sink or source.

Let $KQ/I$ be a monomial algebra, that is, a bound quiver algebra such that the admissible ideal $I$ is generated by paths.
A path $p$ in $I$ is called {\it minimal} if any path $q$ in $I$ such that $p=q^{\prime}q q^{\prime\prime}$ for some $q^{\prime}, q^{\prime\prime}\in\B$ equals $p$.
Let $\F$ be the set of minimal paths in $I$.
Then $\F$ is finite, since $I$ is finitely generated.
For $p\in\B$, when there is no danger of confusion, we simply write the canonical image $p+I$ in $KQ/I$ as $p$.
The following lemma characterizes Cohen-Macaulay monomial algebras $\L$ with $\findim\L=0$ in terms of quivers and minimal relations.

\begin{lemma}\label{claim_3}
Let $\L=KQ/I$ be a non-simple connected monomial algebra and $\F$ the set of minimal paths in $I$.
Then the following conditions are equivalent.
\begin{enumerate}
\item $\findim\L=0=\findim\L^\op$.
\item For each $i\in Q_0$, there exist minimal paths $f_1=a_1\cdots a_n$ and $f_2=b_1\cdots b_m$ in $\F$ with $a_j, b_j \in Q_1$ such that $s(a_n) = i = t(b_1)$.  
\end{enumerate}
\end{lemma}
\begin{proof}
By  \cref{Chan-Iyama-Marczinzik_2024_Cohen-Macaulay_Example1.3(1)_corollary}, it suffices to show that (2) is equivalent to the condition, denoted by $(*)$, that $\soc\L_\L$ and $\soc{}_\L\L$ contain all simple modules.
First, if $(*)$ holds, then $Q$ has no sink or source, and hence (2) holds.
Conversely, assume (2) holds.
For each $i\in Q_0$, take minimal paths $f_1=a_1\cdots a_n, f_2=b_1\cdots b_m\in\F$ with $a_j, b_j \in Q_1$ such that $s(a_n) = i = t(b_1)$.  
Then the simple modules $e_i\L/\rad e_i\L$ and $\L e_i/\rad\L e_i$ are contained in the socles of the indecomposable projective modules $e_{s(f_1)}\L$ and $\L e_{t(f_2)}$, respectively. 
This completes the proof.
\end{proof}

Recall that a {\it truncated path algebra of a quiver $Q$} is a monomial algebra of the form $KQ/R^{L+1}$ for some $L\geq1$, where $R=R_Q$ denotes the ideal of $KQ$ generated by all arrows in $Q$.
Note that $R^{L+1}$ is the ideal of $KQ$ generated by all paths of length $L+1$.
We are now ready to state the first main result of this paper.

\begin{theorem} \label{claim_4}
A truncated path algebra $\L = KQ/R^{L+1}$ of a connected quiver $Q$ is Cohen-Macaulay if and only if $Q$ is balanced. 
In this case, a dualizing module $W$ is given as follows:
\begin{align}
W= 
\begin{cases}
\L&\text{if $Q$ is acyclic,}\\
D(\L)&\text{otherwise.}\\
\end{cases}
\end{align}
In particular, 
\begin{align}
\findim\L= 
\begin{cases}
\gldim\L & \text{if $Q$ is acyclic,}\\
0 & \text{otherwise.}\\
\end{cases} 
\end{align}
\end{theorem}
\begin{proof}
If $\L$ is a Cohen-Macaulay algebra with dualizing module $W$, then, since $D(W)$ is a strong tilting module on each side,  \cite[Corollary 7.2]{Dugas-Huisgen--Zimmermann_2011} concludes that $Q$ has no precyclic source or postcyclic sink, that is, $Q$ is balanced.
Conversely, assume that $Q$ is balanced.
Then $Q$ has a sink or a source if and only if $Q$ is acyclic if and only if $\gldim\L<\infty$.
In this case, $\L$ is clearly Cohen-Macaulay. 
Also, if $\gldim\L=\infty$,
then, since $Q$ has no sink or source, \cref{claim_3} concludes that $\L$ is a Cohen-Macaulay algebra with $\findim\L=0$.
The last statement is now clear.
\end{proof}

\begin{remark}
The theorem shows that there is no non-trivial Cohen-Macaulay truncated path algebra. 
\end{remark}

As a consequence of \cref{claim_4}, we obtain the following characterization of (not necessarily connected) Cohen-Macaulay truncated path algebras.

\begin{corollary}
\label{claim_10}
A truncated path algebra $\L = KQ/R^{L+1}$ of a  quiver $Q$ is  Cohen-Macaulay if and only if $Q$ is balanced.
In this case, the $\L$-bimodule 
\begin{align}
W= \prod_{Q^\prime} KQ^\prime/R^{L+1} \oplus \bigoplus_{Q^{\prime\prime}} D\!\left(KQ^{\prime\prime}/R^{L+1}\right) 
\end{align}
gives a dualizing module, where $Q^\prime$ (resp.~$Q^{\prime\prime}$) runs over the acyclic connected components of $Q$ (resp.~the connected components of $Q$ without  sink or source). 
In particular, 
\begin{align}
\findim\L= 
\begin{cases}
\gldim\prod_{Q^\prime} KQ^\prime/R^{L+1} & \text{if such a $Q^\prime$ exists,}\\
0 & \text{otherwise.}\\
\end{cases} 
\end{align}
\end{corollary}

\begin{example}
Consider a truncated path algebra $\L = KQ/R^{L+1}$, 
where $Q$ is the following (disconnected) quiver:
\begin{align}
Q: \ \xymatrix{
    1 \ar[r] & 2 \ar[r] & 3 
}\hspace{10mm}
\xymatrix{
    4 \ar[r] \ar@(ul,dl) & 5 \ar[r]  & 6 \ar@(ur,dr) \\
} 
\end{align}
Let $Q^1$ and $Q^2$ be the connected components of $Q$ such that $Q^1_0=\{1, 2, 3\}$ and $Q^2_0=\{4, 5, 6\}$.
For $i=1, 2$, let $\L_i :=KQ^i/R^{L+1}$.
Note that $\L=\L_1\times\L_2$ is non-Iwanaga-Gorenstein, since $\projdim_\L D(\L_2)=\infty$.
Since $Q$ is balanced,  \cref{claim_10} implies that $\L$ is a Cohen-Macaulay algebra such that  $W=\L_1\oplus D(\L_2)$ is a dualizing module, and 
\begin{align}
\findim\L
= \gldim\L_1
= \begin{cases}
    2 &\text{if $L=1$},\\
    1 &\text{if $L\geq 2$}.
\end{cases}
\end{align}
\end{example}

Let $\L$ be an algebra.
For a $\L$-module $X$ with minimal injective resolution $0\to X \to I^0(X)\to I^1(X) \to \cdots$, its {\it dominant dimension} $\domdim_\L X$ is defined as the minimal $n \geq 0$ such that $ I^n(X)\not\in\proj\L$ or as infinite if no such $n$ exists.
The {\it dominant dimension} of $\L$  is defined as $\domdim \L:=\domdim_\L \L$.
By \cite[Theorem 4]{Mueller_1968}, we have $\domdim \L = \domdim \L^\op$. 
For $n\geq 1$, we call $X$ {\it $n$-Gorenstein} if $\projdim_\L I^i(X) \leq i$ for $1\leq i <n$.
From the definition, every $n$-Gorenstein module is $m$-Gorenstein for all $m\leq n$; and $X$ is $1$-Gorenstein if and only if $\domdim_\L X \geq 1$.
By \cite[Theorem 3.7]{Fossum-Griffith-Reiten_2006}, $\L$ is $n$-Gorenstein if and only if $\L^\op$ is $n$-Gorenstein.
Moreover, by \cite[Proposition 2.6]{Chan-Iyama-Marczinzik_2024_Cohen-Macaulay}, if $\L$ is a Cohen-Macaulay algebra with dualizing module $W$, then $\L$ is $n$-Gorenstein if and only if $W$ is $n$-Gorenstein.
In this case, $\projdim_\L I^i(\L) =  \projdim_\L I^i(W)$, so that $\domdim_\L  W = \domdim \L = \domdim \L^\op = \domdim_{\L^\op} W$, which we write as $\domdim W$.
The following proposition characterizes $1$-Gorenstein Cohen-Macaulay truncated path algebras in terms of ordinary quivers.

\begin{proposition}\label{claim_9}
A truncated path algebra $\L = KQ/R^{L+1}$
is $1$-Gorenstein Cohen-Macaulay if and only if $\L$ is Nakayama.
In this case, $\L$ is Iwanaga-Gorenstein.
\end{proposition}
\begin{proof}
The \lq\lq if\rq\rq~part follows from \cref{claim_10} and the fact that the ordinary quiver of a Nakayama algebra is a disjoint union of linear and cyclic quivers (cf.~\cite{Abrar_2012_PhD}).
Conversely, assume that $\L$ is a Cohen-Macaulay algebra such that the injective envelope $I$ of a dualizing module $W$, viewed as a right $\L$-module, is projective.
By \cref{claim_10}, we may assume that $W= \prod_{Q^\prime} KQ^\prime/R^{L+1} \oplus \bigoplus_{Q^{\prime\prime}} D\!\left(KQ^{\prime\prime}/R^{L+1}\right)$, where $Q^\prime$ (resp.~$Q^{\prime\prime}$) runs over the acyclic connected components of $Q$ (resp.~ the connected components of $Q$ without  sink or source).
Since $I\in\proj\L$, it follows that $D\!\left(KQ^{\prime\prime}/R^{L+1}\right)\in\proj\L$ for all $Q^{\prime\prime}$.
Then, together with the fact that every indecomposable projective-injective module over a monomial algebra is uniserial, this implies that each of the algebras $KQ^{\prime\prime}/R^{L+1}$ of infinite global dimension is a self-injective Nakayama algebra.
Moreover, for the same reason, each of the algebras $KQ^{\prime}/R^{L+1}$ of finite global dimension 
is both left and right serial, that is, Nakayama.
Therefore, $\L=\prod_{Q^\prime} KQ^\prime/R^{L+1} \times \prod_{Q^{\prime\prime}} KQ^{\prime\prime}/R^{L+1}$ is a Nakayama algebra.
The last statement is clear.
\end{proof}

Following \cite{Iyama-Solberg_2018}, we call an algebra $\L$ {\it minimal Auslander-Gorenstein} if $\injdim \L \leq d \leq \domdim \L$ for some $d \geq 2$.
In this case, $\L$ is Iwanaga-Gorenstein (\cite[Proposition 4.1]{Iyama-Solberg_2018}).
Following \cite{Chan-Iyama-Marczinzik_2024_Cohen-Macaulay}, we call a Cohen-Macaulay algebra $\L$ with dualizing module $W$ {\it minimal Auslander-Cohen-Macaulay} if $\injdim W \leq d \leq \domdim W$ for some $d \geq 2$.
Clearly, minimal Auslander-Gorenstein algebras are minimal Auslander-Cohen-Macaulay.

\begin{corollary}\label{claim_11}
The following conditions are equivalent for a truncated path algebra $\L$.
\begin{enumerate}
\item $\L$ is a minimal Auslander-Cohen-Macaulay algebra.
\item $\L$ is a minimal Auslander-Gorenstein Nakayama algebra.
\end{enumerate}
\end{corollary}
\begin{proof} 
\cref{claim_9} shows the implication (1) $\Rightarrow$ (2).
The converse is clear.
\end{proof}



\subsection{Cohen-Macaulay Nakayama algebras} 

We need no lemma to obtain the following second main result of this paper, which shows that every Cohen-Macaulay Nakayama algebra is Iwanaga-Gorenstein.

\begin{theorem}\label{claim_1}
A Nakayama algebra $\L$ is Cohen-Macaulay if and only if it is Iwanaga-Gorenstein.
\end{theorem}
\begin{proof}
It suffices to show the \lq\lq only if\rq\rq~part.
Suppose that $\L$ is a Cohen-Macaulay algebra with dualizing module $W$.
Fix a complete set $\{P_1, \ldots, P_n\}$ of pairwise non-isomorphic indecomposable projective $\L$-modules. 
Since $W_\L/\rad W_\L$ contains all simple $\L$-modules by \cref{Preliminaries_1}, it follows that
\begin{align}
    W_\L = (P_1/\rad^{r_1}\!P_1)^{s_1} \oplus \cdots \oplus (P_n/\rad^{r_n}\!P_n)^{s_n},
\end{align}
where $r_i$ and $s_i$ are positive integers, and $X^t$ denotes the direct sum of $t$ copies of a $\L$-module $X$.
Assume $r_i < m_i=l(P_i)$ (i.e.~$\rad^{r_i}\!P_i\neq 0$) for some $1\leq i \leq n$.
Since $\L$ is Nakayama, we have $\soc P_i = \rad^{m_i-1}P_i=x\L$ for some $x\in\L$.
Then the natural algebra isomorphism $\varphi:\L\xrightarrow{\sim}\End(W_\L)$ sends the generator $x$ of the simple socle of $P_i$ to a generator $\varphi(x)$ of the simple socle of the indecomposable projective $\End(W_\L)$-module 
$\Hom_\L(W, P_i/\rad^{r_i}\!P_i)$.
However, since $x = 0$ in $P_i/\rad^{r_i}\!P_i$ by assumption, we have $\varphi(x)(w) = x w = 0$ for all $w\in W$, and hence $\varphi(x)=0$, a contradiction.
Therefore, $r_i \geq m_i=l(P_i)$ (i.e.~$\rad^{r_i}\!P_i=0$)  for all $1\leq i \leq n$.
Now, since $W_\L=\L_\L\in\proj\L$, \cref{claim_8} concludes that $\L$ is Iwanaga-Gorenstein.
\end{proof}

\begin{remark}\label{remark_2}
\begin{enumerate}
\item The theorem shows that there is no non-Iwanaga-Gorenstein Cohen-Macaulay Nakayama algebra.
\item The theorem remains true for the case where Nakayama algebras are artin.
\item In \cite[Proposition 5]{Ringel_2013}, Ringel characterized Iwanaga-Gorenstein Nakayama algebras in terms of resolution quivers.
\end{enumerate}
\end{remark}



\section*{Acknowledgments}
The second author would like to thank Osamu Iyama, Aaron Chan and Yuya Otake for kindly answering his questions.
The suggestions by the first two helped improve the paper.
The first author was partially supported by JSPS KAKENHI Grant Number 26K16972.
The second author was partially supported by JSPS KAKENHI Grant Number 25K17243.

\bibliographystyle{plain}
\bibliography{ref}

\end{document}